\documentclass[a4paper,12pt]{article}
\usepackage{amsmath}
\usepackage{amsfonts}
\usepackage{supertabular}
\usepackage[latin9]{inputenc}
\usepackage[english]{varioref}
\usepackage{dcolumn}
\usepackage[height=22cm , width = 16cm , top = 4cm , left = 3cm, a4paper]{geometry}
\usepackage{amsthm}
\usepackage{dsfont}
\usepackage{hyperref}
\usepackage[hypcap]{caption}
\usepackage{color}

\numberwithin{equation}{section}

\theoremstyle{plain}
\newtheorem{thm}{Theorem}[section]

\newtheorem{lem}[thm]{Lemma}
\newtheorem{prop}[thm]{Proposition}

\newtheorem{rmkk}[thm]{Remark}

\newcommand{\enter}{\bigskip}

\date{}

\begin{document}
\author{Prasanta Kumar Barik${}^1$\footnote{{\it{${}$ Email address:}} prasant.daonly01@gmail.com/pbarik@tifrbng.res.in}, Ankik Kumar Giri${}^2$ and Rajesh Kumar${}^3$\\
\footnotesize ${}^1$Tata Institute of Fundamental Research Centre for Applicable Mathematics,\\ \small{Bangalore-560065, Karnataka, India}\\
\footnotesize ${}^2$Department of Mathematics, Indian Institute of Technology Roorkee,\\ \small{ Roorkee-247667, Uttarakhand, India}\\
\footnotesize ${}^3$Department of Mathematics, Birla Institute of Technology and Science, Pilani,\\ \small{ Pilani-333031, Rajasthan, India}\\
}

\title {Mass-conserving weak solutions to the coagulation and collisional breakage equation with singular rates }

\maketitle

\hrule \vskip 8pt

\begin{quote}
{\small {\em\bf Abstract.} In this article, the existence of mass-conserving solutions is investigated to the continuous coagulation and collisional breakage equation with singular coagulation kernels. Here, the probability distribution function attains singularity near the origin. The existence result is constructed by using both conservative and non-conservative truncations to the continuous coagulation and collisional breakage equation. The proof of the existence result relies on a classical weak $L^1$ compactness method.
\enter
}
\end{quote}

\hspace{.5cm}

\vskip 10pt \hrule

\section{Introduction}\label{existintroduction1}
The coagulation and fragmentation models are a particular class of partial integro-differential equation in which two particles collide at a particular instant to form a larger particle by aggregation process or split into more than two fragments by nonlinear breakage process or a large particle splits into many particles by linear breakage process. Here, we consider a fully non-linear partial integro-differential equation i.e., continuous coagulation and collisional breakage equation (CCBE) where each particle is fully identified by its volume (or size) $v \in \mathds{R}_{+}:=(0, \infty)$. In recent papers \cite{Barik:2018Weak, Barik:2018Existence}, we have discussed the existence and uniqueness of weak solutions to continuous coagulation and collisional breakage model. This model has a lot of applications in different field of science, engineering and technology such as astrology and astrophysics. In this article, we focus on the case of existence of mass-conserving solutions to continuous CCBE. The general continuous CCBE reads as \cite{Brown:1995, Safronov:1972, Vigil:2006, Wilkins:1982}
\begin{eqnarray}\label{CCBE}
\frac{\partial g}{\partial t}  =\mathcal{C}_{B}(g)- \mathcal{CB}_{D}(g)  +\mathcal{B}_{B}(g),
\end{eqnarray}
where
\begin{eqnarray*}
\mathcal{C}_{B}(g)(v, t) := \frac{1}{2}\int_0^v  E(v-v', v') \varphi(v-v', v') g(v-v', t) g(v', t)dv',
\end{eqnarray*}
\begin{eqnarray*}
\hspace{-2.8cm} \mathcal{CB}_{D}(g)(v, t) : = \int_{0}^{\infty}  \varphi(v, v') g(v, t) g(v', t) dv',
\end{eqnarray*}
and
 \begin{align*}
\mathcal{B}_{B}(g) (v, t) :=  & \frac{1}{2} \int_{v}^{\infty} \int_{0}^{v'}  P(v|v'-v''; v'')E_1(v'-v'', v'') \nonumber\\
 &~~~~~~~~~~~~~~~\times \varphi(v'-v'', v'') g(v'-v'', t) g(v'', t)dv'' dv',
\end{align*}
with initial data
\begin{align}\label{Initialdata}
g(v, 0) = g^{in}(v) \ge 0\ \ \mbox{a.e.}.
\end{align}
Here the unknown $g(v, t) \ge 0$ is a function of volume variable $v \in \mathds{R}_{+}$ and the time variable $t \ge 0$ and is known as the concentration of particles. The coefficient $\varphi(v, v')$ is called the collision kernel which represents the coagulation rate which is non-negative and symmetric function  and describes the rate at which particles of volume $v$ and $v'$ interact and produce the larger particles of volume $v+v'$ with coalescence efficiency $E(v, v')$ and the breakage efficiency $E_1(v, v')$. Here, $E(v, v')+E_1(v, v')=1$. The probability distribution function $P(v|v';v'')$ gives the birth of particles of volume $v$ with the collision between particles of volumes $v'$ and $v''$. In addition, the transfer of volumes between particles $v'$ and $v''$ may occur. Furthermore, the probability distribution function $P$ is assumed to enjoy the following properties
\begin{align}\label{TNP}
&\int_0^{v'+v''} P(v|v';v'')dv =T_N(v', v'')\quad \forall \quad (v', v'') \in \mathds{R}_{+}^2,\nonumber\\
\text{where}&\ \underset{(v', v'')\in \mathds{R}_{>0}^2 }{\sup} T_N(v', v'')=T_N <\infty, \quad P(v|v';v'')dv=0 \quad \forall \quad v \ge v'+v'',
\end{align}
and
\begin{align}\label{MCP}
\int_0^{v'+v''} v P(v|v';v'')dv =v'+v''\quad \forall \quad v\in (0, v'+v'').
\end{align}
In equation \eqref{TNP}, $T_N(v', v'')$ is the total number of daughter particles obtained due to the collision between particles of volume $v'$ and $v''$. We have considered its supremum as $T_N$ which is a positive constant greater than or equal to 2 and equation (\ref{MCP}) shows the conservation of matter in the system during the collisional breakage events. However, the total matter may not conserve during the nonlinear coagulation and nonlinear breakage processes due to appearance of infinite gel in the system. This happens due to the high aggregation rate in compare to the fragmentation rate. This physical phenomena is known as gelation transition and the least time at which this happens is called as \emph{gelation time} \cite{Escobedo:2003, Leyvraz:1981}.\\

The total mass of particles in the system for coagulation and collisional breakage equation can be defined as
 \begin{align}\label{Totalmass}
\mathcal{N}_1(t)=\mathcal{N}_1(g(t)):=\int_0^{\infty} v g(v, t)dv,  \ \ t \ge 0.
\end{align}
In particular, if $t=0$, the total mass of particles in the system can be represented by the following notation:
 \begin{align}\label{Totalinitialmass}
\mathcal{N}_1^{in}=\mathcal{N}_1(g(0)):=\int_0^{\infty} v g^{in}(v)dv.
\end{align}

In equation \eqref{CCBE}, the first term $\mathcal{C}_B$ shows the birth of new particles of volume $v+v'$ due to the collision between particles of volumes $v$ and $v'$ through the aggregation process while the second term $\mathcal{CB}_D$ gives the death of particles of volume $v$  due to both coagulation and collisional breakage events. The last term $\mathcal{B}_B$ represents the formation of particles of volume $v$ due to the collisional breakage process.\\

The existence and uniqueness of mass-conserving weak solutions to the continuous coagulation and linear fragmentation equation with both nonsingular and singular kernels have been extensively studied in several articles, see \cite{Barik:2017Anote, Barik:2019, Barik:2018Mass, Camejo:2015, Giri:2012, Laurencot:2018, McLaughlin:1997} and references therein. However, there are only a few articles available in which the nonlinear fragmentation equation is described, see \cite{Cheng:1990, Cheng:1988, Ernst:2007, Kostoglou:2000, Laurencot:2001}. In \cite{Cheng:1990, Cheng:1988, Ernst:2007}, the authors have investigated the scaling solutions to the continuous nonlinear fragmentation equation whereas the analytic solutions for the different cases of kernels are discussed in \cite{Ernst:2007, Kostoglou:2000}. Later, the existence of weak solutions and the asymptotic behaviour of solutions to the discrete version of non-linear fragmentation equation are studied in \cite{Laurencot:2001}. To the best of our knowledge, the fully nonlinear homogeneous continuous coagulation and collisional breakage equation is described by Safronov first time, see \cite{Safronov:1972}. Later Wilkins \cite{Wilkins:1982} gave the geometrical interpretation of this model. Recently, we have studied the existence of weak solutions to the continuous CCBE with both non-singular and singular collision kernels by using conservative truncation to the original equation, see \cite{Barik:2018Weak, Barik:2018Existence}, where the singular collision kernel to the continuous CCBE is
\begin{equation*}
\varphi(v, v') \le k\frac{ (1+v)^{\beta} (1+v')^{\beta} } {(v+v')^{\alpha}},\ \beta \in [0, 1),\  \alpha \in (0, 1/2)\ \text{and}\ \beta-\alpha \in [0, 1).
\end{equation*}
 In addition, a uniqueness result is studied for a special case of collision kernel when $\beta=0$. However, the mass-conservation property of the weak solution is not investigated. Thus, it is clear that there is no uniqueness result available to the continuous CCBE for singular coagulation kernel which is stated in assumption $(A_1)$ in the next section. Hence, there may be some solutions which are either mass-conserving or not conserving the mass. Due to non-availability of the uniqueness result to the continuous CCBE, both conservative and non-conservative approximations are considered in this work. As we know a conservative approximations may always give mass-conserving solution whereas a non-conservative truncation is suitable to study the gelation transition, an obvious question arise is that whether a non-conservative approximation can give mass-conserving solution or not? The purpose of the present work is to provide a positive answer of this question as well as the existence of mass-conserving weak solution by using a conservative approximation to the original problem. The motivation of the present work is taken from \cite{Barik:2018Mass, Barik:2019, Filbet:2004Mass}\\

The content of the paper: we describe some assumptions on collision kernel, probability distribution function and coalesce efficiency in Section 2. Furthermore, the main result and some preliminary results for the convex function are stated in Section 3. The proof of the existence of mass-conserving weak solutions is shown by using a weak $L^1$ compactness method in this section.

 \section{Assumptions, Statement of Preliminaries and Main Result}
In this section, some assumptions on collision kernel $\varphi$, distribution function $P$, and the coalescence efficiency $E$ are stated.\\
\textbf{Assumptions:}
$(A_1)$ Let $\varphi$ be a non-negative measurable function on $\mathds{R}_{+} \times \mathds{R}_{+}$, and it satisfies $\varphi(v, v') \le k \frac{(1+v+v')}{(v+v')^{\alpha}}$ for all $(v, v') \in \mathds{R}_{+} \times \mathds{R}_{+}$, $0< \alpha  < \frac{1}{2}$ and for some constant $k > 0$,\\
\\
$(A_2)$ $E$ satisfies the following condition locally:
\begin{eqnarray}\label{probabilitys}
 E(v, v') \ge \frac{\eta(2 \alpha)-2}{\eta(2 \alpha)-1},\quad \forall\quad (v, v') \in (0, 1) \times (0, 1),
\end{eqnarray}
\\
$(A_3)$  there exists a positive constant $\eta(2 \alpha)>2$ (depending on $\theta$ and $\alpha$) such that
\begin{eqnarray*}
\int_0^{v'+v''} v^{-2\alpha} P(v|v';v'')dv \le \eta(2 \alpha) (v'+v'')^{-2\alpha},
\end{eqnarray*}
where $P(v|v';v'')=(\theta+2)\frac{v^{\theta}}{(v'+v'')^{1+\theta}}$, for $-1< \theta \le 0$ and $\alpha$ \& $\theta$ are related with the relation $2\alpha -\theta <1$.\\
\\
($A_4$) $g^{in} \in L^1_{-2\alpha, 1}(\mathds{R}_{+})$.\\

Next, the main result of this paper is provided.
\begin{thm}\label{TheoremCCBE}
Assume that $(A_1)$--$(A_4)$ hold. Let $g_n$ be the solution to \eqref{CCBETRUN}(defined later on) for $n \ge 1$. Then there exists a subsequence $(g_{n_k})$ of $(g_{n})$ and a mass conserving solution $g$ to \eqref{CCBE}--\eqref{Initialdata} such that
 \begin{align}\label{TheoremEquation}
 g_{n_k}\to g \ \  \text{in}\ \mathcal{C}([0,T]^w; L_{-\alpha, 1}^1(\mathds{R}_{+} ) )\ \text{for each}\ T>0
 \end{align}
 and $g$ satisfying the weak formulation
 \begin{align}\label{definition}
\int_0^{\infty} \{ g(v, t) & - g^{in}(v) \} h(v)dv = \frac{1}{2}\int_0^t \int_0^{\infty} \int_{0}^{\infty} \tilde{h}(v, v') E(v, v') \varphi(v, v') g(v, s) g(v', s)dv' dv ds\nonumber\\
&+\frac{1}{2} \int_0^t \int_0^{\infty} \int_0^{\infty} \Pi_{h}(v', v'') E_1(v', v'')  \varphi(v', v'')  g(v', s) g(v'', s) dv'' dv' ds,
\end{align}
where
\begin{align}\label{Identity1}
\tilde{h} (v, v')=h(v+v')-h(v)-h(v')
\end{align}
and
\begin{align}\label{Identity2}
 \Pi_{h}(v', v'') =\int_0^{v'+v''}  h(v) P(v|v'; v'')dv-h(v')-h(v''),
\end{align}
for every $t \in [0, T]$ and $ h  \in L^{\infty}(\mathds{R}_{+})$. Here, $\mathcal{C}( [0,T]^w;  L^1_{-\alpha, 1}(\mathds{R}_{+}) )$ denotes the space of all weakly continuous functions from $[0, T]$ to $L_{-\alpha, 1}^1(\mathds{R}_{+})$. In addition, a sequence $(g_n)$ converges to $g$ in $\mathcal{C}( [0,T]^w;  L^1_{-\alpha, 1}(\mathds{R}_{+}))$ if
\begin{align}
\lim_{n \to \infty} \sup_{t \in [0,T]}\bigg|\int_0^{\infty}(v^{-\alpha}+v)[g_n(v, t)-g(v, t)] h(v) dv \bigg|=0,
\end{align}
for every $h \in L^{\infty}(\mathds{R}_{+})$.
\end{thm}
Next, let us define a particular class of convex functions $\mathcal{C}_{VP, \infty}$. Let $\Gamma_1, \Gamma_2  \in \mathcal{C}^{\infty}([0, \infty))$ be two non-negative and convex functions, then $\Gamma_1, \Gamma_2  \in \mathcal{C}_{VP, \infty} $ if
\begin{description}
  \item[(a)] $\Gamma_l(0)=\Gamma_l'(0)=0$ and $\Gamma_l'$ is concave;
  \item[(b)] $\lim_{s \to \infty} \Gamma_l'(s) =\lim_{s \to \infty} \frac{ \Gamma_l(s)}{s}=\infty $;
  \item[(c)] for $\gamma \in (1, 2)$,
  \begin{align*}
  S_{\gamma}(\Gamma_l):= \sup_{s \ge 0} \bigg\{   \frac{\Gamma_l( s )}{s^{\gamma}} \bigg\} < \infty.
  \end{align*}
 for $l=1, 2$.
\end{description}

Since $g^{in}\in L_{-2\alpha, 1}^1(\mathds{R}_{+})$, then a refined version of de la Vall\'{e}e-Poussin theorem, see  \cite[Theorem~2.8]{Laurencot:2015}, ensures that there exist two non-negative functions $\Gamma_1$ and $\Gamma_2$ in $\mathcal{C}_{VP, \infty}$ with
\begin{align}\label{convexp1}
\Gamma_l(0)=0,~~~\lim_{q \to {\infty}}\frac{\Gamma_l(q)}{q}=\infty,~~~~l=1,2
\end{align}
and
\begin{align}\label{convexp2}
\Upsilon_1 := \int_0^{\infty} \Gamma_1(v)g^{in}(v)dv<\infty,~~\text{and}~\Upsilon_2 :=\int_0^{\infty}{\Gamma_2(v^{-\alpha}g^{in}(v))}dv <\infty.
\end{align}
Finally, some additional properties of $\mathcal{C}_{VP, \infty}$ which are also required to prove Theorem \ref{TheoremCCBE} are discussed.
\begin{lem} Consider $\Gamma_1$, $\Gamma_2$ in $\mathcal{C}_{VP, \infty}$. Then we have the following results
\begin{equation}\label{convexp3}
\hspace{-5cm} \Gamma_2(q_1)\le q_1\Gamma'_2(q_1)\le 2\Gamma_2(q_1),
\end{equation}
\begin{equation}\label{convexp4}
\hspace{-5.5cm} q_1\Gamma_2'(q_2)\le \Gamma_2(q_1)+\Gamma_2(q_2),
\end{equation}
and
\begin{equation}\label{convexp5}
0 \le \Gamma_1(q_1+q_2)-\Gamma_1(q_1)-\Gamma_1(q_2)\le  2\frac{q_1\Gamma_1(q_2)+q_2\Gamma_1(q_1)}{q_1+q_2},
\end{equation}
 for all $q_1, q_2 \in \mathds{R}_{+}$.
\end{lem}
\begin{proof}
This lemma can easily be proved in a similar way as given in \cite{Barik:2017Anote, Barik:2018Mass, Laurencot:2018}.
\end{proof}

\section{Existence of weak solutions}
This section deals with the construction of mass conserving solution for the conservative and non-conservative truncations to \eqref{CCBE}--\eqref{Initialdata}. It is expected that a mass conserving solution can be obtained for the conservative approximation under some restricted kernels. Moreover, considering a non-conservative form of coagulation and a conservative approximation of multiple fragmentation is appropriate to study the gelation transition. Therefore, an obvious question arises whether such coupling of mixed approximations will provide a mass-conserving solution or not to \eqref{CCBE}--\eqref{Initialdata}? Interestingly, the answer of this question is yes and we provide the proof of this result in this work.\\

Let us define here both the conservative and non-conservation approximations to \eqref{CCBE}--\eqref{Initialdata}.
For a given natural number $n\in\mathds{N}$, we set
\begin{align}\label{Initialtrunc}
g_n^{in}(v)=g^{in}(v) \mathds{1}_{(0, n)}(v),
\end{align}
and for  $\tau \in \{0, 1\}$,
\begin{equation}\label{CoagKerTrun}
\varphi_n^\tau (v, v') := \varphi(v, v')  \mathds{1}_{(1/n, n)}(v)  \mathds{1}_{(1/n, n)}(v') \left\{ 1 - \tau + \tau  \mathds{1}_{(0, n)}(v+v') \right\}.
\end{equation}
Using \eqref{Initialtrunc} and \eqref{CoagKerTrun}, the equations \eqref{CCBE}--\eqref{Initialdata} can be rewritten as
\begin{align}\label{CCBETRUN}
\frac{\partial g_n}{\partial t}  =\mathcal{C}_B^n(g_n)-\mathcal{CB}_D^n(g_n)+\mathcal{B}_B^n(g_n),
\end{align}
with initial data
\begin{align}\label{Initialdata1}
g_n(v, 0) = g_n^{in}(v)\ge 0\ \ \mbox{a.e.},
\end{align}
where
\begin{eqnarray*}\label{C1trun}
\mathcal{C}_B^n (g_n)(v, t) := \frac{1}{2}\int_0^v E(v-v', v' ) \varphi_n^{\tau} (v-v', v')g_n(v-v', t) g_n(v', t)dv',
\end{eqnarray*}
 \begin{align*}
\mathcal{B}_B^n (g_n)(v, t) :=  & \frac{1}{2} \int_{v}^{n} \int_{0}^{v'}P(v|v'-v''; v'')  E_1(v'-v'', v'')\} \nonumber\\
 &~~~~~~~~~~~~~~~\times \varphi_n^{\tau}(v'-v'', v'') g_n(v'-v'', t) g_n(v'', t)dv'' dv',
\end{align*}
and
\begin{eqnarray*}\label{B2}
\mathcal{CB}_D^n(g_n)(v, t) : = \int_{0}^{n-\tau v}  \varphi_n^{\tau} (v, v') g_n(v, t) g_n(v', t) dv'.
\end{eqnarray*}
The additional variable $\tau \in \{0, 1\}$ permits us to handle simultaneously the conservative approximation $(\tau = 1)$ and non-conservative approximation $(\tau = 0)$.\\

Now, in the following Proposition, the result on the positivity of the unique solution is summarized followed by a Remark on conservative and non-conservative approximation.
\begin{prop}\label{propBarik}
Let $\tau\in \{0,1\}$ and $n\ge 1$. Then, there exists a unique non-negative solution $g_n\in \mathcal{C}^1([0,\infty);L^1(0,n))$ to \eqref{CCBETRUN}--\eqref{Initialdata1}. In addition, it satisfies
\begin{align}\label{Barikprop}
\int_0^n  v g_n(v, t) dv  = &\int_0^n v g^{in}_n(v) dv \nonumber\\
&  - (1-\tau) \int_0^t \int_0^n \int_{n- v}^n v E(v, v') \varphi_n^{\tau}(v, v') g_n(v, s)g_n(v', s)dv' dv ds
\end{align}
for $t\ge 0$.
\end{prop}
\begin{proof}
The proof of the Proposition \ref{propBarik} is similar to \cite[Proposition 4.1]{Barik:2018Weak}.
\end{proof}

\begin{rmkk}
It should be mentioned here that the last term on the right-hand side of \eqref{Barikprop} leads to the case of total mass conservation for $\tau =1$ and therefore known as the conservative approximation. While for $\tau =0$, the truncative form becomes non-conservative as the total mass decreases with respect to time. Moreover, in both these cases, it is clear that
\begin{align}\label{Barikprop1}
\int_0^n  v  g_n(v, t)dv  \le \int_0^n v g^{in}_n(v) dv,\ \ \ \text{for}\ \ t > 0.
\end{align}
\end{rmkk}

Further, the weak formulation to \eqref{CCBETRUN}--\eqref{Initialdata1} for $n\ge 1$ and $h \in L^{\infty}(\mathbb{R}_{+})$ can be obtained as
\begin{align}\label{nctp3}
& \int_0^n \{ g_n (v, t)-g_n^{in}(v)\} h(v)dv \nonumber\\
= & \frac{1}{2} \int_0^t\int_0^n \int_0^n \tilde{h}_{\tau}(v, v') \mathds{1}_{(1/n, n)}(v) \mathds{1}_{(1/n, n)}(v') E(v, v') \varphi(v, v') g_n(v, s)g_n(v', s)dv' dv ds \nonumber\\
 +&\frac{1}{2} \int_0^t \int_0^n\int_0^{n} \Pi_{h, \tau}(v', v'')  \mathds{1}_{(1/n, n)}(v') \mathds{1}_{(1/n, n)}(v'') E_1(v', v'') \nonumber\\
& ~~~~~~~~~\times  \varphi(v', v'') g_n(v', s)g_n(v'', s) dv'' dv' ds,
\end{align}
where
\begin{align}\label{G Omega}
 \tilde{h}_{\tau}(v, v')=  h(v+v')  \mathds{1}_{(0, n)}(v+v')- \{h(v)+h(v') \} (1-\tau +\tau  \mathds{1}_{(0, n)}(v+v'))
 \end{align}
and
\begin{align}\label{PIZ}
\Pi_{h, \tau}(v', v'')= & \mathds{1}_{(0, n)}(v'+v'') \int_0^{v'+v''} h(v) P(v|v'; v'')dv''\nonumber\\
& - \{h(v')+h(v'')\} (1-\tau +\tau  \mathds{1}_{(0, n)}(v'+v'')).
\end{align}

Finally, we claim and later prove that the family of solutions $\{ g_n\}_{n \ge 1}$ is relatively compact in $\mathcal{C}([0,T]^w; L_{-\alpha, 1}^1(\mathds{R}_{+} ) )$. For this purpose, the weak $L^1$ compactness method is applied which is also used in the pioneering work of Stewart \cite{Stewart:1989} and Lauren\c{c}ot et. al. \cite{Laurencot:2002}. To proceed further, the uniform boundedness of the family of solutions $\{ g_n\}_{n \ge 1}$ in space $L_{-2\alpha, 1}^1(\mathds{R}_{+})$ is shown in the next section.

\subsection{Uniform Bound}
\begin{lem}\label{B(T)}
Let us assume $(A_1)$--$(A_4)$ hold. Let $T>0$. Then there exists a positive constant $ \mathcal{B}(T)$ depending on $T$ such that
\begin{align*}
\int_0^{n} (v^{-2\alpha}+v)g_n(v, t)dv \le \mathcal{B}(T)\ \ \text{for all}\ t \in[0, T].
\end{align*}
\end{lem}
\begin{proof}
 Let $t\in [0,T]$ and $n \ge 1$. Multiplying the equation \eqref{CCBETRUN} by $v^{-2\alpha}$ and then taking the integration from $0$ to $n$ with respect to $v$ leads to
\begin{align}\label{sEquibound2}
\frac{d}{dt}\int_0^n v^{-2\alpha} g^n(v, t)dv =&\int_0^n v^{-2\alpha}  [\mathcal{C}_B^n(g^n)(v, t)
-\mathcal{CB}_{D}^n(g^n)(v, t)+\mathcal{B}_B^n (g^n)(v, t)] dv.
\end{align}
Now, we estimate each integral on the right-hand side of (\ref{sEquibound2}), individually. The first term on the right-hand side of (\ref{sEquibound2}) can be simplified, by using Fubini's theorem, the transformation $v-v'=v_1$ $\&$ $v'=v_2$, and then replacing the variables $v_1$ by $v'$ and $v_2$ by $v$ as
\begin{align}\label{sEquibound3}
\int_0^n v^{-2\alpha}  \mathcal{C}_B^n(g^n)(v, t)dv  
 =\frac{1}{2} \int_0^n \int_{0}^{n-v}  (v+v')^{-2\alpha} E(v, v')  \varphi_n^{\tau}(v, v')g^n(v, t)g^n(v', t)dv' dv.
\end{align}
Again, applying the repeated application of Fubini's theorem to the third term on the right-hand side of (\ref{sEquibound2}), using the relation $(A_3)$, the transformation $v'-v''=v_2$ $\&$ $v''=v_3$ and finally replacing $v' \rightarrow v$ $\&$ $v'' \rightarrow v'$, we get
\begin{align}\label{sEquibound4}
\int_0^n v^{-2\alpha}  \mathcal{B}_B^n(g^n)(v, t)dv  =&\frac{1}{2}  \int_0^n  \int_{v}^{n} \int_0^{v'}  v^{-2\alpha} P(v| v'-v''; v'') E_1(v'-v''; v'') \varphi_n^{\tau} (v'-v'', v'') \nonumber\\
&~~~~~~~~~~~~\times  g^n(v'-v'', t)g^n(v'', t) dv'' dv' dv \nonumber\\
=&\frac{1}{2}  \int_0^n \int_0^{v'} \int_0^{v'} v^{-2\alpha} P(v| v'-v''; v'') E_1(v'-v''; v'')  \varphi_n^{\tau} (v'-v'', v'') \nonumber\\
&~~~~~~~~~~~\times g^n(v'-v'', t)g^n(v'', t) dv dv'' dv' \nonumber\\
=&\frac{1}{2}  \int_0^n \int_0^{n-v''} \int_0^{v'+v''} v^{-2\alpha} P(v| v'; v'') E_1(v', v'') \varphi_n^{\tau} (v', v'') \nonumber\\
&~~~~~~~~~~~\times  g^n(v', t)g^n(v'', t) dv dv' dv'' \nonumber\\
\le & \frac{\eta(2\alpha)}{2}   \int_0^n \int_0^{n-v} (v+v')^{-2\alpha} E_1(v, v') \varphi_n^{\tau}(v, v') g^n(v, t)g^n(v', t) dv' dv.
\end{align}
Further, inserting (\ref{sEquibound3}) and (\ref{sEquibound4}) into (\ref{sEquibound2}), and using Fubini's theorem and $(A_3)$,  (\ref{sEquibound2}) can be rewritten as
\begin{align}\label{sEquibound6}
\frac{d}{dt}  \int_0^n  v^{-2\alpha}  g^n (v, t) dv &
\le  - \int_0^1\int_{0}^{1-v} \bigg[1-\frac{1}{2}  E(v, v') -\frac{\eta(2\alpha)}{2} E_1(v, v') \bigg] \nonumber\\
 &~~~~~~~~~~~~~~~~~~~~~~~~~\times v^{-2\alpha} \varphi_n^{\tau}(v, v')g^n(v, t) g^n(v', t)dv' dv \nonumber\\
&- \int_0^1\int_{1-v}^{1} \bigg[1-\frac{\eta(2\alpha)}{2} E_1(v, v') \bigg] v^{-2\alpha}\varphi_n^{\tau} (v, v') g^n(v, t) g^n(v', t)dv' dv \nonumber\\
&+ \int_1^n \int_{0}^{n}  (v+v')^{-2\alpha} E(v, v') \varphi_n^{\tau}(v, v')g^n(v, t)g^n(v', t)dv' dv\nonumber\\
&- \int_0^1 \int_{1}^{n-v \tau} v^{-2\alpha} \varphi_n^{\tau} (v, v') g^n(v, t) g^n(v', t) dv' dv \nonumber\\
&- \int_1^n \int_{0}^{n-v \tau} v^{-2\alpha} \varphi_n^{\tau} (v, v') g^n(v, t) g^n(v', t) dv' dv \nonumber\\
&+ \frac{\eta(2\alpha)}{2}  \int_0^1  \int_{1}^{n} v^{-2\alpha}  E_1(v, v')  \varphi_n^{\tau} (v, v') g^n(v, t) g^n(v', t) dv' dv \nonumber\\
& + \frac{\eta(2\alpha)}{2}  \int_1^n  \int_{0}^{n}  v^{-2\alpha}  E_1(v, v')  \varphi_n^{\tau} (v, v') g^n(v, t) g^n(v', t) dv' dv.
\end{align}
Using the negativity of the first, second, fourth and fifth integrals on the right-hand side of (\ref{sEquibound6}) which are guaranteed from $(A_2)$ and Proposition \ref{propBarik}, and then further applying $(A_1)$, we obtain
\begin{align*}
\frac{d}{dt}  \int_0^n  v^{-2\alpha}  g^n (v, t) dv &
\le \int_1^n \int_{0}^{n}  (v+v')^{-2\alpha}  \varphi_n^{\tau}(v, v') g^n(v, t) g^n(v', t)dv' dv\nonumber\\
&+ \frac{\eta(2\alpha)}{2}  \int_0^1  \int_{1}^{n} v^{-2\alpha} \varphi_n^{\tau}(v, v') g^n(v, t) g^n(v', t)dv' dv \nonumber\\
& + \frac{\eta(2\alpha)}{2}  \int_1^n  \int_{0}^{n} v^{-2\alpha}   \varphi_n^{\tau}(v, v') g^n(v, t) g^n(v', t)dv' dv \nonumber\\
\le & 2k \bigg(1+ \frac{\eta(2\alpha)}{2} \bigg) \int_1^n \int_{0}^{n} v^{-2\alpha} \frac{ (1+v+v')}{(v+v')^{\alpha}}g^n(v, t)g^n(v', t)dv' dv \nonumber\\
&+ k \frac{\eta(2\alpha)}{2}   \int_0^1  \int_{1}^{n} v^{-2\alpha} \frac{ (1+v+v')}{(v+v')^{\alpha}}g^n(v, t)g^n(v', t)dv' dv \nonumber\\
\le & 2k \bigg(1+ \frac{\eta(2\alpha)}{2} \bigg)  \int_1^n \int_{0}^{n}  (v+v') g^n(v, t)g^n(v', t)dv' dv \nonumber\\
&+ 2k\ \eta(2\alpha)    \int_0^1  \int_{1}^{n} v^{-2\alpha}  v' g^n(v, t)g^n(v', t)dv' dv
\end{align*}
Finally, by using \eqref{Totalinitialmass}, the above can be further reduces to
\begin{align}\label{sEquibound7}
\frac{d}{dt}  \int_0^n  v^{-2\alpha}  g^n (v, t) dv\le 4 k \mathcal{ N}_1^{in}  (1+   \eta(2\alpha) ) \int_{0}^{n}  v^{-2\alpha} g^n(v, t) dv  + 2  k { \mathcal{N}_1^{in}}^2 (2+ \eta(2\alpha) ).
\end{align}
Now, an application of Gronwall's inequality and $(A_4)$ to (\ref{sEquibound7}) gives
\begin{eqnarray}\label{sEquibound8}
\int_0^n v^{-2\alpha} g^n(v, t) dv \le \mathcal{B}_1(T),
\end{eqnarray}
where
\begin{eqnarray*}
\mathcal{B}_1(T):=e^{aT} \|g_{\text{in}}\|_{L^1_{-2\alpha, 1}(\mathds{R}_+)} +
\frac{b }{a} ( e^{aT } -1 ),
\end{eqnarray*}
for $a:=  4 k_1  \mathcal{N}_1^{in}  (1+   \eta(2\alpha) ) $ and $b:= 2  k_1{ \mathcal{N}_1^{in}}^2 (2+ \eta(2\alpha) )$.
Hence, using (\ref{sEquibound8}) yields
 \begin{eqnarray*}
\int_0^{n} (v+v^{-2\alpha}) g^n(v, t)dv \le \mathcal{B}_1(T)+ \mathcal{N}_1^{in}:= \mathcal{B}(T),
\end{eqnarray*}
which completes the proof of Lemma \ref{B(T)}.
\end{proof}

In the following Lemma, we discuss the behaviour of $g_n$ for large volume particle $v$.
\begin{lem}\label{MassConCoagMulti}
Let $(A_1)$--$(A_4)$ hold. Then for every $n > 1$, $T>0$,
\begin{equation}\label{C(T)}
\sup_{t\in [0, T]}\int_0^n \Gamma_1 (v)g_n(v, t)dv  \le \mathcal{G}(T),
\end{equation}
and
\begin{align}\label{C(T1)}
(1-\tau)\int_0^T& \int_0^n\int_{n-v}^n \Gamma_1 (v)  \mathds{1}_{(1/n, n)}(v)  \mathds{1}_{(1/n, n)}(v')\nonumber\\
 &~~~~~~~~~~~\times \varphi(v, v')g_n(v, s)g_n(v', s) dv' dv ds\le \mathcal{G}(T),
\end{align}
where the $\Gamma_1 \in \mathcal{C}_{VP, \infty}$ satisfies \eqref{convexp1} and \eqref{convexp2} and $\mathcal{G}(T)$ (depending on $T$) is a positive constant.
\end{lem}
\begin{proof}   Choose $h(v)=\Gamma_1(v) \chi_{(0, n)}(v)$, and inserting it into (\ref{nctp3}) to obtain
\begin{align}\label{large1}
& \int_0^n \Gamma_1(v) g_n(v, t)dv =\int_0^n \Gamma_1(v)g_n^{in}(v)dv \nonumber\\
& + \frac{1}{2} \int_0^t \int_0^n \int_0^n \tilde{\Gamma}_{1, \tau}(v, v')  \mathds{1}_{(1/n, n)}(v)  \mathds{1}_{(1/n, n)}(v') E(v, v') \varphi(v, v') g_n(v, s)g_n(v', s) dv' dv ds\nonumber\\
& +\frac{1}{2} \int_0^t \int_0^n \int_0^{n} \Pi_{\Gamma_1, \tau} (v', v'')  \mathds{1}_{(1/n, n)}(v')  \mathds{1}_{(1/n, n)}(v'') E_1(v', v'') \nonumber\\
 &~~~~~~~~~~~~ \times \varphi(v', v'') g_n(v', s)g_n(v'', s)dv'' dv' ds,
\end{align}
where
\begin{align}\label{large2}
 \tilde{\Gamma}_{1, \tau}(v, v') = \Gamma_1 (v+v')  \mathds{1}_{(0, n)}(v+v')-\{\Gamma_1 (v)+\Gamma_1 (v')\}  (1-\tau +\tau  \mathds{1}_{(0, n)}(v+v'))
 \end{align}
and
\begin{align*}
\Pi_{\Gamma_1, \tau} (v', v'') = &  \mathds{1}_{(0, n)}(v'+v'') \int_0^{v'+v''} \Gamma_1(v) P(v|v';v'')dv\nonumber\\
 &-\{ \Gamma_1(v') +\Gamma_1(v'')\} (1-\tau +\tau  \mathds{1}_{(0, n)}(v'+v'')).
\end{align*}
By using \eqref{convexp2} and \eqref{Initialtrunc} into \eqref{large1}, we achieve
\begin{align}\label{large3}
\int_0^n\Gamma_1(v)g_n(v, t)dv \le &\Upsilon_1+\frac{1}{2}\int_0^t \{ S_1^n(s)+S_2^n(s)+S_3^n(s)+S_4^n(s) \} ds,
\end{align}
where
\begin{align*}
S_1^n(s)=& \int_0^n \int_0^{n-v}  \{ \Gamma_1 (v+v')-\Gamma_1 (v)-\Gamma_1 (v') \}  \mathds{1}_{(1/n, n)}(v)  \mathds{1}_{(1/n, n)}(v')E(v, v') \nonumber\\
& ~~~~~~~~~ \times  \varphi(v, v') g_n(v, s)g_n(v', s) dv' dv,\nonumber\\
S_2^n(s)=&-(1-\tau) \int_0^n \int_{n-v}^n \{\Gamma_1 (v)+\Gamma_1 (v')\}   \mathds{1}_{(1/n, n)}(v)  \mathds{1}_{(1/n, n)}(v')E(v, v') \nonumber\\
& ~~~~~~~~~ \times  \varphi(v, v') g_n(v, s)g_n(v', s) dv' dv,\nonumber\\
S_3^n(s)=& \int_0^n \int_0^{n-v'}  \bigg[\int_0^{v'+v''} \Gamma_1(v) P(v|v';v'')dv -\{ \Gamma_1(v') + \Gamma_1(v'')\}   \bigg]   \mathds{1}_{(1/n, n)}(v) \mathds{1}_{(1/n, n)}(v')\nonumber\\
& ~~~~~~~~~ \times E_1(v, v')  \varphi(v', v'') g_n(v', s)g_n(v'', s) dv'' dv',
\end{align*}
and
\begin{align*}
S_4^n(s)= &-(1-\tau )\int_0^n \int_{n-v'}^n  \{ \Gamma_1(v') + \Gamma_1(v'')\}   \mathds{1}_{(1/n, n)}(v)  \mathds{1}_{(1/n, n)}(v') E_1(v, v')\nonumber\\
  &~~~~~~~~\times \varphi(v', v'') g_n(v', s)g_n(v'', s) dv'' dv'.
\end{align*}
Now, we simplify each term separately. The part $S_1^n(s)$ is estimated by using \eqref{convexp5} and ($A_1$) as
\begin{align}\label{Pn}
S_1^n(s) \le &   \int_0^n \int_0^{n}  \{ \Gamma_1 (v+v')-\Gamma_1 (v)-\Gamma_1 (v') \}   \varphi(v, v')   \mathds{1}_{(1/n, n)}(v)  \mathds{1}_{(1/n, n)}(v')  \nonumber\\
  &~~~~~~~~~~~~~~\times g_n(v, s)g_n(v', s) dv' dv \nonumber\\
\le & 2k \int_0^1 \int_0^{1}\frac{(1+v+v')}{(v+v')^{\alpha}} \times  \frac{ v\Gamma_1(v')+ v' \Gamma_1(v)}{(v+v')} g_n(v, s)g_n(v', s) dv' dv \nonumber\\
& + 4k \int_1^n \int_0^{1} \frac{(1+v+v')}{(v+v')^{\alpha}} \times  \frac{ v\Gamma_1(v')+ v' \Gamma_1(v)}{(v+v')}  g_n(v, s)g_n(v', s) dv' dv \nonumber\\
& + 2k \int_1^n \int_1^{n}  \frac{(1+v+v')}{(v+v')^{\alpha}} \times  \frac{ v\Gamma_1(v')+ v' \Gamma_1(v)}{(v+v')} g_n(v, s)g_n(v', s) dv' dv \nonumber\\
\le & 12 k \int_0^1 \int_0^{1} v^{-2\alpha}   \frac{ v \Gamma_1(v') }{(v+v')} g_n(v, s)g_n(v', s) dv' dv  \nonumber\\
& + 12 k \int_1^n \int_0^{1}\frac{v}{(v+v')^{\alpha}} \times  \frac{ v\Gamma_1(v')+v'\Gamma_1(v)}{(v+v')} g_n(v, s)g_n(v', s) dv' dv  \nonumber\\
& + 4k \int_1^n \int_1^{n}  \frac{v\Gamma_1(v')+v'\Gamma_1(v)}{(v+v')^{\alpha}} g_n(v, s)g_n(v', s) dv' dv.
\end{align}
Let us first estimate the first term on the right-hand side of \eqref{Pn}. By using Lemma \ref{B(T)} and the monotonicity of $\Gamma_1$, it can be simplified as
\begin{align}\label{Pn1}
12 k \int_0^1 \int_0^{1} v^{-2\alpha}   \frac{v\Gamma_1(v') }{(v+v')} g_n(v, s)g_n(v', s)dv' dv
 \le 12 k \Gamma_1(1) \mathcal{B}(T)^2.
\end{align}
Again considering Lemma \ref{B(T)} and the monotonicity of $\Gamma_1$, the second term on the right-hand side of \eqref{Pn} can be evaluated as
\begin{align}\label{Pn2}
12 k \int_1^n \int_0^{1} & \frac{v}{(v+v')^{\alpha}} \times  \frac{v \Gamma_1(v')+ v' \Gamma_1(v)}{(v+v')}  g_n(v, s)g_n(v', s)dv' dv \nonumber\\
\le & 12 k \Gamma_1(1) \int_1^n \int_0^{1} v v'^{-2\alpha}  g_n(v, s)g_n(v', s)dv' dv  \nonumber\\
&+12 k \int_1^n \int_0^{1}  v'^{-2\alpha}  \Gamma_1(v) g_n(v, s)g_n(v', s)dv' dv \nonumber\\
\le & 12 k \Gamma_1(1) \mathcal{B}(T)^2 +12 k \mathcal{B}(T) \int_0^n   \Gamma_1(v) g_n(v, s) dv.
\end{align}
Finally, the last integral on the right-hand side of \eqref{Pn} is calculated by applying Lemma \ref{B(T)} as
\begin{align}\label{Pn3}
 4k \int_1^n \int_1^{n}  \frac{ v\Gamma_1(v')+v'\Gamma_1(v)}{(v+v')^{\alpha}} g_n(v, s)g_n(v', s)dv' dv
\le & 8 k \int_1^n \int_1^{n} v \Gamma_1(v')  g_n(v, s)g_n(v', s)dv' dv\nonumber\\
\le & 8 k \mathcal{B}(T) \int_0^n  \Gamma_1(v) g_n(v, s) dv.
\end{align}
Inserting \eqref{Pn1}, \eqref{Pn2} and \eqref{Pn3} into \eqref{Pn} lead to
\begin{align}\label{Pnf}
S_1^n(s) \le & 24 k \Gamma_1(1)\mathcal{B}(T)^2 + 20 k \mathcal{B}(T) \int_0^n   \Gamma_1(v) g_n(v, s) dv.
\end{align}
Next, the term $S_2^n(s)$ can be estimated as
\begin{align}\label{Qn}
S_2^n(s)=-(1-\tau) \int_0^n \int_{n-v}^n & \{\Gamma_1 (v)+\Gamma_1 (v')\}  \mathds{1}_{(1/n, n)}(v)  \mathds{1}_{(1/n, n)}(v') E(v, v') \nonumber\\
 &~~~~~~  \times  \varphi(v, v') g_n(v, s)g_n(v', s) dv' dv \le  0
\end{align}
due to the non-negativity of $g_n$ and $\varphi$.\\

To proceed further in estimating the term $S_3^n(s)$, the monotonicity of ${\Gamma'_1}$ together with the relations
\eqref{MCP} and \eqref{convexp5} are used which give
\begin{align*}
S_3^n(s) \le & \int_0^n \int_0^{n}  \bigg[\int_0^{v'+v''} \frac{\Gamma_1(v)}{v} v P(v|v';v'')dv - \Gamma_1(v') - \Gamma_1(v'') \bigg] \nonumber\\
&~~~~~~~~~~~~~~~\times \varphi(v', v'') g_n(v', s)g_n(v'', s) dv'' dv' \nonumber\\
\le & \int_0^n \int_0^{n}  \bigg[\frac{\Gamma_1(v'+v'')}{(v'+v'')} \int_0^{v'+v''}  v P(v|v';v'')dv - \Gamma_1(v') - \Gamma_1(v'') \bigg] \nonumber\\
&~~~~~~~~~~~~~~~~~\times  \varphi(v', v'') g_n(v', s)g_n(v'', s) dv'' dv'\nonumber\\
= & \int_0^n \int_0^{n}  \{ \Gamma_1(v'+v'') -\Gamma_1(v') - \Gamma_1(v'') \}  \varphi(v', v'') g_n(v', s)g_n(v'', s) dv'' dv'.
 \end{align*}
Further, by using $(A_1)$, the above can be simplified as
\begin{align*}
S_3^n(s)
\le & 2k \int_0^1 \int_0^{1} \frac{(1+v'+v'')}{(v'+v'')^{\alpha}} \times  \frac{v' \Gamma_1(v'')+ v'' \Gamma_1(v')}{(v'+v'')} g_n(v', s)g_n(v'', s) dv'' dv' \nonumber\\
& + 4k \int_1^n \int_0^{1}  \frac{(1+v'+v'')}{(v'+v'')^{\alpha}} \times  \frac{v' \Gamma_1(v'')+ v'' \Gamma_1(v')}{(v'+v'')}  g_n(v', s)g_n(v'', s) dv'' dv'\nonumber\\
& + 2k \int_1^n \int_1^{n}   \frac{(1+v'+v'')}{(v'+v'')^{\alpha}} \times  \frac{v' \Gamma_1(v'')+ v'' \Gamma_1(v')}{(v'+v'')} g_n(v', s)g_n(v'', s) dv'' dv' \nonumber\\
\le & 12 k \int_0^1 \int_0^{1} v'^{-2\alpha}   \frac{v' \Gamma_1(v'') }{(v'+v'')} g_n(v', s)g_n(v'', s) dv'' dv'  \nonumber\\
& + 12 k \int_1^n \int_0^{1}\frac{v'}{(v'+v'')^{\alpha}} \times  \frac{v' \Gamma_1(v'')+ v''\Gamma_1(v')}{(v'+v'')}g_n(v', s)g_n(v'', s) dv'' dv' \nonumber\\
& + 4k \int_1^n \int_1^{n}  \frac{v' \Gamma_1(v'')+ v'' \Gamma_1(v')}{(v'+v'')^{\alpha}} g_n(v', s)g_n(v'', s) dv'' dv'.
\end{align*}
Hence, by applying Lemma \ref{B(T)}, we get
\begin{align}\label{Rn}
S_3^n(s) \le & 24 k \Gamma_1(1)\mathcal{B}(T)^2 + 20 k \mathcal{B}(T) \int_0^n   \Gamma_1(v) g_n(v, s) dv.
\end{align}
Finally, $S_4^n(s)$ can be estimated similar to $S_2^n(s)$ as
\begin{align}\label{Sn}
S_4^n(s)\le  &-2(1-\tau )\int_0^n \int_{n-v'}^n  \Gamma_1(v')     \mathds{1}_{(1/n, n)}(v)  \mathds{1}_{(1/n, n)}(v') E_1(v, v')\nonumber\\
  &~~~~~~~~\times \varphi(v', v'') g_n(v', s)g_n(v'', s) dv'' dv' \le 0.
\end{align}
Inserting (\ref{Pnf}), (\ref{Qn}), (\ref{Rn}) and (\ref{Sn}) into (\ref{large3}) gives us
\begin{align*}
\int_0^n \Gamma_1(v)g_n(v, t)dv &+ 2k(1-\tau) \int_0^t \int_0^n \int_{n-v}^{n} \Gamma_1(v)  \mathds{1}_{(1/n, n)}(v)  \mathds{1}_{(1/n, n)}(v')\nonumber\\
  &~~~~~~~~~~~~~~~~~~~\times  \varphi(v, v') g_n(v, s) g_n(v', s)dv' dv ds \nonumber\\
  \le & \Upsilon_1 + 48 k \Gamma_1(1)\mathcal{B}(T)^2 + 40 k \mathcal{B}(T) \int_0^n   \Gamma_1(v) g_n(v, s)dv.
\end{align*}
Then by Gronwall's inequality, we get
\begin{align*}
\int_0^n \Gamma_1(v)g_n(v, t)dv &+ (1-\tau) \int_0^t \int_0^n\int_{n-v}^{n} \Gamma_1(v)  \mathds{1}_{(1/n, n)}(v)  \mathds{1}_{(1/n, n)}(v')\nonumber\\
  &~~~~~~~~~~~~~~~~~~~\times  \varphi(v, v') g_n(v, s) g_n(v', s)dv' dv ds
 \le  \mathcal{G}(T),
\end{align*}
where $\mathcal{G}(T)=  ( \Upsilon_1 + 48 k \Gamma_1(1)\mathcal{B}(T)^2T ) e^{40 k \mathcal{B}(T) T}$, and it completes the proof of Lemma \ref{MassConCoagMulti}.
\end{proof}

To proceed further in achieving our main goal, a refined version of de la Vall\`{e}e Poussin theorem \cite{Laurencot:2015} is used to show the equi-integrability condition for the family of solutions $\{ g_n\}_{n > 1}$ in the next subsection.

\subsection{Equi-integrability}
\begin{lem}\label{last}
Assume that the kinetic coefficient $\varphi$, the coalescence efficiency $E$, the probability distribution function $P$ and the initial data satisfy $(A_1)$--$(A_4)$, respectively. Let $ \lambda \in (1, n)$ and $T>0$, then
\begin{align*}
 \sup_{t\in [0,T]}\int_0^{\lambda} \Gamma_2( v^{-\alpha}  g_n(v, t))dv \le \Xi(\lambda, T),
\end{align*}
where $\Gamma_2 \in \mathcal{C}_{VP, \infty}$ satisfies \eqref{convexp1} and \eqref{convexp2} and $\Xi(\lambda, T)$ is a positive constant depending on $\lambda$ and $T$.\\
\end{lem}
\begin{proof} We first consider $h_n(v, t):= v^{-\alpha}  g_n(v, t)$ and $n > \lambda$. Then by using \eqref{CCBETRUN} and non-negativity of $g_n$, we obtain
\begin{align}\label{Equint1}
\frac{d}{dt}\int_0^{\lambda} \Gamma_2(h_n(v, t))dv \le &\frac{1}{2} \int_0^{\lambda} \int_0^v \Gamma_2'(h_n(v, t))  v^{-\alpha} \varphi_n^{\tau}(v-v', v') g_n(v-v', t)g_n(v', t)dv' dv \nonumber\\
&+\frac{1}{2} \int_0^{\lambda} \int_v^n \int_0^{v'}  \Gamma_2'(h_n(v, t))  v^{-\alpha}  P(v|v'-v''; v'') \nonumber\\
& ~~~~~~~~~~~~\times  \varphi_n^{\tau}(v'-v'', v'') g_n(v'-v'', t) g_n(v'', t)dv'' dv' dv.
\end{align}
Next, the first term on the right-hand side to \eqref{Equint1} can be estimated, by using Fubini's theorem, applying the transformation $v-v'=v_1$ and $v'=v_2$, $(A_1)$, \eqref{convexp4} and Lemma \ref{B(T)}, as
\begin{align}\label{Equintp3}
 \frac{1}{2} \int_0^{\lambda} \int_0^{\lambda-v_2} & \Gamma_2'(h_n(v_1+v_2, t)) (v_1+v_2)^{-\alpha} \varphi_n^{\tau}(v_1, v_2) g_n(v_1, t)g_n(v_2, t) dv_1 dv_2 \nonumber\\
 \le & \frac{1}{2}k (1+\lambda) \int_0^\lambda \int_0^{\lambda-v'}  v^{-\alpha} v'^{-\alpha}  \Gamma_2'(h_n(v+v', t))g_n(v, t) g_n(v', t) dv dv' \nonumber\\
\le & \frac{1}{2}k (1+\lambda) \int_0^\lambda \int_0^{\lambda-v'}  v'^{-\alpha} \{ \Gamma_2(h_n(v+v', t))+\Gamma_2 (h_n(v, t)) \} g_n(v', t)dv dv' \nonumber\\
\le &  A(\lambda, T) \int_0^\lambda \Gamma_2 (h_n(v, t))dv,
\end{align}
where $A(\lambda, T):= k(1+\lambda) \mathcal{B}(T)$. Now, we evaluate the second term, by using repeated application of Fubini's theorem, substituting $v'-v''=v_2$ and $v''=v_3$, as
\begin{align}\label{Equintp3}
& \frac{1}{2} \int_0^\lambda \int_0^{v_2} \int_0^{v_2}   \Gamma_2'(h_n(v, t))  v^{-\alpha}  P(v|v_2-v_3;v_3) \varphi_n^{\tau}(v_2-v_3, v_3) g_n(v_2-v_3, t) g_n(v_3, t)dv dv_3   dv_2 \nonumber\\
&+\frac{1}{2} \int_\lambda^n  \int_0^{v_2} \int_0^\lambda \Gamma_2'(h_n(v, t))  v^{-\alpha}  P(v|v_2-v_3;v_3) \varphi_n^{\tau}(v_2-v_3, v_3) g_n(v_2-v_3, t) g_n(v_3, t)dv dv_3   dv_2 \nonumber\\
=&\frac{1}{2} \int_0^\lambda \int_0^{\lambda-v''} \int_0^{v'+v''}  \Gamma_2'(h_n(v, t))  v^{-\alpha}  P(v|v';v'') \varphi_n^{\tau}(v', v'') g_n(v', t) g_n(v'', t)dv dv' dv''  \nonumber\\
&+\frac{1}{2} \int_0^\lambda  \int_{\lambda-v''}^{n-v''} \int_0^\lambda \Gamma_2'(h_n(v, t))  v^{-\alpha}  P(v|v';v'') \varphi_n^{\tau}(v', v'') g_n(v', t) g_n(v'', t)dv dv' dv'' \nonumber\\
&+\frac{1}{2} \int_\lambda^n  \int_0^{n-v''} \int_0^\lambda \Gamma_2'(h_n(v, t))  v^{-\alpha}  P(v|v';v'') \varphi_n^{\tau}(v', v'') g_n(v', t) g_n(v'', t)dv dv' dv'' \nonumber\\
\le &\frac{1}{2} \int_0^n  \int_{0}^{n} \int_0^\lambda \Gamma_2'(h_n(v, t))  v^{-\alpha}  P(v|v';v'') \varphi_n^{\tau}(v', v'') g_n(v', t) g_n(v'', t)dv dv' dv''.
\end{align}
By using $(A_3)$, \eqref{convexp4}, the definition of $\mathcal{C}_{VP, \infty}$ and Lemma \ref{B(T)}, we estimate \eqref{Equintp3} as
\begin{align}\label{est1}
&\frac{1}{2}k (\theta+2) \int_0^n  \int_{0}^{n} \int_0^\lambda  \{ \Gamma_2(h_n(v, t))+ \Gamma_2( v^{\theta-\alpha} ) \} (1+v'+v'') \frac{1}{(v'+v'')^{\theta+1}} (v'+v'')^{-\alpha} \nonumber\\
& ~~~~~~~~~~~~~~ \times  g_n(v', t) g_n(v'', t)dv dv' dv''\nonumber\\
\le & \frac{1}{2}k (\theta+2) \int_0^n  \int_{0}^{n} \int_0^\lambda  \Gamma_2(h_n(v, t))  \frac{(1+v')+(1+v'')}{v'^{\theta+1}} v'^{-\alpha} g_n(v', t) g_n(v'', t)dv dv' dv''\nonumber\\
 & + \frac{1}{2}k (\theta+2) \int_0^n  \int_{0}^{n} \int_0^\lambda   \Gamma_2( v^{\theta-\alpha} ) \frac{(1+v')(1+v'')}{v'^{\theta+1}} v'^{-\alpha}  g_n(v', t) g_n(v'', t)dv dv' dv''\nonumber\\
 \le & k (\theta+2) \mathcal{B}(T) \bigg[ \int_0^{\lambda}  \Gamma_2(h_n(v, t)) dv  +  \int_0^{\lambda}  v^{\gamma(\theta-\alpha)} \frac{\Gamma_2( v^{\theta-\alpha} ) }{ v^{\gamma(\theta-\alpha)} }  dv  \bigg] \nonumber\\
  &~~~~~~~~~~~~~~~~~\times\bigg(\int_0^1  \frac{(1+v')}{v'^{\theta+1}} v'^{-\alpha} g_n(v', t) dv'+\int_1^n  \frac{(1+v')}{v'^{\theta+1}} v'^{-\alpha} g_n(v', t) dv' \bigg)  \nonumber\\
   \le & 2 k (\theta+2) \mathcal{B}(T)  \bigg[ \int_0^{\lambda}  \Gamma_2(h_n(v, t)) dv  + S_{\gamma}(\Gamma_2)   \frac{\lambda^{\gamma(\theta-\alpha)+1}}{ \gamma(\theta-\alpha)+1}  \bigg] \nonumber\\
  &~~~~~~~~~~~~~~~~~\times\bigg(\int_0^1  v'^{-2\alpha} g_n(v', t) dv'+\int_1^n  v g_n(v', t) dv' \bigg)  \nonumber\\
    \le& 2k (\theta+2) \mathcal{B}^2(T) \bigg[ \int_0^{\lambda}  \Gamma_2(h_n(v, t)) dv  + S_{\gamma}(\Gamma_2)   \frac{\lambda^{\gamma(\theta-\alpha)+1}}{ \gamma(\theta-\alpha)+1}  \bigg].
\end{align}
Collecting estimates in \eqref{Equintp3} and \eqref{est1}, and inserting them into \eqref{Equint1}, it gives
\begin{align}\label{est2}
\frac{d}{dt}\int_0^\lambda \Gamma_2(h_n(v, t))dv \le  A^{\ast}(\lambda, T) \int_0^\lambda \Gamma_2(h_n(v, t)) dv +A^{\dag}(\lambda, T),
\end{align}
where $ A^{\ast}(\lambda, T):= A(\lambda, T)+2k(\theta +2) \mathcal{B}(T)$ and $A^{\dag}(\lambda, T):= 2k \frac{(\theta +2)}{(\gamma \theta-\gamma \alpha+1)} S_\gamma(\Gamma_2)  \mathcal{B}^2(T)\lambda^{\gamma \theta -\gamma \alpha+1} $.
Finally, using \eqref{convexp2} and the Gronwall's inequality into \eqref{est2}, we obtain
\begin{align*}
\int_0^\lambda \Gamma_2(v^{-\alpha} g_n(v, t))dv \le \Xi(\lambda, T).
\end{align*}
 This proves Lemma \ref{last}.
\end{proof}

\subsection{Equi-continuity w.r.t. time in weak sense}
\begin{lem}\label{Equicontinuityweaksense}
Let $T>0$ and $ \lambda \in (1, n)$. Assume $(A_1)$--$(A_4)$ hold. Then for $0 \le s \le t \le T$ and $\Delta \in L^{\infty}(\mathds{R}_{+})$, we have
\begin{align*}
\bigg|\int_0^{\lambda}  v^{-\alpha} \Delta(v) [g_n(v, t)-g_n(v, s)]dv \bigg|\le \Theta(\lambda, T)(t-s),
\end{align*}
where $\Theta(\lambda, T)$  is a positive constant depending on $\lambda$ and $T$.
\end{lem}
\begin{proof}
Let $0 \le s \le t \le T$ such that $|t-s|<\delta$ and $\delta = \frac{\epsilon}{\Theta(\lambda, T)}$, and $\Delta \in L^{\infty}(\mathds{R}_{+})$. Next, we simplify the following integral as
\begin{align}\label{Equicontinuity1}
\int_0^\lambda v^{-\alpha} & \Delta(v)  |g_n(v, t)-g_n(v, s)|dv \nonumber\\
 \le & \| \Delta \|_{L^{\infty}(\mathds{R}_{+})} \int_s^t \int_0^\lambda v^{-\alpha} \bigg|  \frac{\partial g_n}{\partial t}(v, \zeta) \bigg| dv d\zeta \nonumber\\
 \le  & \| \Delta \|_{L^{\infty}(\mathds{R}_{+})} \int_s^t  \bigg[   \frac{1}{2} \int_0^\lambda  \int_{0}^{v}  v^{-\alpha} \varphi_n^{\tau}(v-v', v') g_n(v-v', \zeta) g_n(v', \zeta)dv' dv \nonumber\\
  &+\frac{1}{2} \int_0^\lambda \int_{v}^{n} \int_{0}^{v'} v^{-\alpha} P(v|v'-v''; v'') \varphi_n^{\tau}(v'-v'', v'') g_n(v'-v'', \zeta) g_n(v'', \zeta)dv'' dv'\nonumber\\ & + \int_0^\lambda  \int_{0}^{n-v\tau} v^{-\alpha} \varphi_n^{\tau}(v, v') g_n (v, \zeta) g_n(v', \zeta)dv' dv \bigg] d\zeta.
\end{align}
Now, we estimate each integral on the right-hand side to \eqref{Equicontinuity1} separately. First, we evaluate the first integral, by using Fubini's theorem, ($A_1$) and Lemma \ref{B(T)}, as
\begin{align}\label{Equicontinuity2}
 \frac{1}{2} \int_s^t \int_0^\lambda  \int_{0}^{v} & v^{-\alpha} \varphi_n^{\tau}(v-v', v') g_n(v-v', \zeta) g_n(v', \zeta)dv' dv d\zeta \nonumber\\
\le  &  \frac{1}{2}k \int_s^t \int_0^\lambda  \int_{0}^{\lambda-v'} (v+v')^{-\alpha} \frac{(1+v+v')}{(v+v')^{\alpha}} g_n(v, \zeta) g_n(v', \zeta) dv dv' d\zeta \nonumber\\
\le  &  \frac{1}{2}k (1+\lambda)\int_s^t \int_0^\lambda  \int_{0}^\lambda v^{-\alpha} v'^{-\alpha} g_n(v, \zeta) g_n(v', \zeta) dv dv' d\zeta \nonumber\\
\le &   \frac{1}{2}k (1+\lambda) \mathcal{B}^2(T) (t-s).
\end{align}
Similar to \eqref{Equicontinuity2}, we estimate the second integral, by applying ($A_1$) and Lemma \ref{B(T)} as
\begin{align}\label{Equicontinuity3}
 \int_s^t \int_0^\lambda & \int_{0}^{n-v\tau}  v^{-\alpha} \varphi^{\tau}_n(v, v') g_n(v, \zeta) g_n(v', \zeta) dv' dv d\zeta \nonumber\\
\le & k  \int_s^t \int_0^\lambda  \int_{0}^{n}  v^{-\alpha}  \frac{(1+\lambda+v')}{(v+v')^{\alpha}}  g_n(v, \zeta) g_n(v', \zeta) dv' dv d\zeta \nonumber\\
\le & k \mathcal{B}(T)  \int_s^t  \int_{0}^{n}  (1+\lambda+v') v'^{-\alpha}  g_n(v', \zeta)dv' d\zeta
\le  2 k (1+\lambda) \mathcal{B}^2(T)(t-s).
\end{align}
The last integral can be estimated, by using the repeated application of Fubini's theorem, the transformation $v'-v''=v_1$ \& $v''=v_2$, ($A_1$), ($A_3$) and Lemma \ref{B(T)}, as
\begin{align}\label{Equicontinuity4}
\frac{1}{2} &  \int_s^t \int_0^\lambda \int_{v}^{n} \int_{0}^{v'} v^{-\alpha} P(v|v'-v''; v'') \varphi_n^{\tau}(v'-v'', v'') g_n(v'-v'', \zeta) g_n(v'', \zeta) dv'' dv' dv d\zeta \nonumber\\
\le & \frac{1}{2}  \int_s^t \int_0^n \int_{0}^{n} \int_{0}^{v'+v''} v^{-\alpha} P(v|v'; v'') \varphi_n^{\tau}(v', v'') g_n(v', \zeta) g_n(v'', \zeta) dv  dv'  dv'' d\zeta \nonumber\\
\le & \frac{1}{2} k \frac{(\theta+2)} {(\theta+1-\alpha)} \int_s^t \int_0^n \int_{0}^{n} (v'+v'')^{-\alpha}  \frac{(1+v'+ v'')}{(v'+v'')^{\alpha}}  g_n(v', \zeta) g_n(v'', \zeta)  dv'  dv'' d\zeta \nonumber\\
= & \frac{1}{2} k \frac{(\theta+2)} {(\theta+1-\alpha)} \int_s^t \bigg\{ \int_0^1 \int_{0}^{1} +\int_0^1 \int_{1}^{n} +\int_1^n \int_{0}^{1}+\int_1^n \int_{1}^{n} \bigg\} \frac{(1+v'+ v'')}{(v'+v'')^{2\alpha}} \nonumber\\
&~~~~~~~~~~~~~~~~~~\times   g_n(v', \zeta) g_n(v'', \zeta)  dv'  dv'' d\zeta \nonumber\\
\le  & \frac{1}{2} k \frac{(\theta+2)} {(\theta+1-\alpha)} \int_s^t \bigg\{ 3\mathcal{B}^2(T) + 3\mathcal{B}^2(T) +3\mathcal{B}^2(T) +4\mathcal{B}^2(T) \bigg\} d\zeta \nonumber\\
\le  & \frac{13}{2} k \frac{(\theta+2)} {(\theta+1-\alpha)} \mathcal{B}^2(T)(t-s).
\end{align}
Inserting \eqref{Equicontinuity2}, \eqref{Equicontinuity3}, and \eqref{Equicontinuity4} into \eqref{Equicontinuity1}, we obtain
\begin{align}\label{Equicontinuity6}
\int_0^\lambda v^{-\alpha} & \Delta(v)  |g_n(v, t)-g_n(v, s)|dv \le  \Theta(\lambda, T) (t-s),
\end{align}
where
\begin{align*}
\Theta(\lambda, T)=\frac{1}{2}k  \| \Delta \|_{L^{\infty}(\mathds{R}_{+})} \bigg( 3 (1+\lambda)  + 13  \frac{(\theta+2)} {(\theta+1-\alpha)}  \bigg) \mathcal{B}^2(T).
\end{align*}
This completes the proof of Lemma \ref{Equicontinuityweaksense}.
\end{proof}

\subsection{Convergence of integrals}
In this section, we complete the proof of Theorem \ref{TheoremCCBE} by the using above subsections.
\emph{Proof of Theorem} \ref{TheoremCCBE}:\
From the refined version of de la Vall\`{e}e Poussin theorem, Lemma \ref{B(T)}--\ref{last}, and then using Dunford-Pettis theorem and a variant of the Arzel\`{a}-Ascoli theorem, see \cite{Vrabie:1995}, we conclude that $(g_n)$ is relatively compact in $\mathcal{C}([0, T]^w; L_{-\alpha}^1(0, \lambda))$ for each $T>0$. Then there exists a subsequence of $(g_n)$ (not relabeled) and a nonnegative function $g\in \mathcal{C}([0, T]^w; L_{-\alpha}^1(0, \lambda))$ such that
\begin{align}\label{weakconvergence1}
 g_n \to g \ \ \ \text{in} \ \ \mathcal{C}([0, T]^w: L^1(0, \lambda); v^{-\alpha} dv)
\end{align}
for each $T>0$. We can improve the convergence \eqref{weakconvergence1} to
\begin{align}\label{weakconvergence2}
 g_n \to g \ \ \ \text{in} \ \ \mathcal{C}([0, T]^w: L^1(\mathds{R}_{>0}); (v^{-\alpha} +v) dv)
\end{align}
by applying Lemma \ref{B(T)}, (\ref{C(T)}) and \eqref{convexp1}.\\
 Next, we need to show $g$ is actually a solution to \eqref{CCBE}--\eqref{Initialdata} in the sense of \eqref{definition}. For this, we have to claim all the truncated integrals in \eqref{CCBETRUN} converges weakly to the original integrals in \eqref{CCBE}, respectively. To prove this convergence of integrals, one can follow \cite{Barik:2018Existence}. Now, using the weak convergence \eqref{weakconvergence2} into \eqref{CCBETRUN}, we have
\begin{align*}
& \int_0^{\infty} h(v)  \{ g(v, t)-g^{in}(v) \} dv = \lim_{n \to \infty} \int_0^n h(v) \{g_n(v, t)-g_n^{in}(v) \} dv \nonumber\\
 = & \lim_{n\to \infty} \bigg\{
\frac{1}{2} \int_0^t\int_0^n \int_0^n \tilde{h}_{\tau}(v, v') \mathds{1}_{(1/n, n)}(v) \mathds{1}_{(1/n, n)}(v') E(v, v') \varphi(v, v') g_n(v, s)g_n(v', s)dv' dv ds \nonumber\\
& +\frac{1}{2} \int_0^t \int_0^n\int_0^{n} \Pi_{h, \tau}(v', v'') E_1(v', v'') \varphi(v', v'') \mathds{1}_{(1/n, n)}(v') \mathds{1}_{(1/n, n)}(v'')  g_n(v', s)g_n(v'', s) dv'' dv' ds \bigg\}\nonumber\\
=&  \frac{1}{2}\int_0^t \int_0^{\infty} \int_{0}^{\infty}\tilde{h}(v, v') E(v, v') \varphi(v, v') g(v, s) g(v', s)dv' dv ds \nonumber\\
&+\frac{1}{2} \int_0^t \int_0^{\infty} \int_0^{\infty} \Pi_{h}(v', v'') E_1(v', v'')  \varphi(v', v'')  g(v', s) g(v'', s) dv'' dv' ds,
\end{align*}
for every $h \in L^{\infty}(\mathds{R}_{+})$. This confirms that $g$ is a weak solution to \eqref{CCBE}--\eqref{Initialdata} in the sense of \eqref{definition}.\\
 Finally for the completeness to the proof of Theorem \ref{TheoremCCBE}, it remains to prove that $g$ is a mass-conserving solution to \eqref{CCBE}--\eqref{Initialdata}. In the one hand, for ($\tau=0$), it can be  easily shown similar to \cite{Filbet:2004Mass, Barik:2017Anote} and on the other hand, for conservative case ($\tau=1$), we infer from \eqref{weakconvergence2} and \eqref{Barikprop}, which completes the proof of Theorem \ref{CCBE}.

%

\bibliographystyle{plain}

\end{document}